\documentclass[12pt]{article}
\usepackage{amsfonts,amsthm}

\usepackage{showkeys}

\newtheorem{theorem}{Theorem}[section]

\newtheorem{lemma}{Lemma}[section]

\newtheorem{remark}{Remark}[section]

\def\R{\mathbf R}

\def\E{\mathbf R^l}

\def\ld#1{#1^{\downarrow}_{H}}
\def\d#1{#1^{\prime}_{H}}

\newcommand{\scalpr}[2]{\langle#1,#2\rangle}
\newcommand{\norm}[1]{\Vert#1\Vert}

\begin{document}

\title{Optimality conditions via exact penalty functions}

\author{Vsevolod Ivanov Ivanov
\\
Department of Mathematics, \\
Technical University of Varna, 9010 Varna, Bulgaria
}

\maketitle

\begin{abstract}
In this paper, we obtain optimality conditions for the problem with inequality,  equality and closed set constraints in terms of the lower Hadamard derivative. The results are obtained applying exact penalty functions. 



\end{abstract}

\section{Introduction}
\label{s1}
Optimality conditions play important role in optimization.

In this article we obtain optimality conditions of Fritz John type for the
problem with inequality, equality and a set constraints via exact penalty
functions. The results are given in terms of lower Hadamard derivative.
Quadratic penalty functions were originally applied in McShane \cite{McS73}
for problems with continuously differentiable data. Our results use the
constraint qualification (\ref{6}). Similar assumption appeared in
Demyanov, Di Pillo and Facchinei \cite{dem98} and references therein.
Both assumptions coincide in the case when the problem has no inequality
and set constraints.

We cite the following papers where optimality conditions were obtained in
terms of the lower Hadamard derivative or the respective subdifferential.
Optimality conditions are given in  
Jimenez and Novo \cite{jim02},
Luu and Nguen \cite{luu05}.
Sufficient conditions for the problem with a set constraint are given in Penot
\cite[Proposition 5.7(a)]{pen84}, for the problem with inequality constraints
and a set constraint - in Ward \cite[Corollary 3.1]{war88}.
We refer to Ioffe \cite[Proposition 6]{iof84} where "fuzzy" Lagrange
multiplier rule were obtained in terms of the lower Hadamard subdifferential
under the name Dini subdifferential.
Necessary optimality conditions were derived in Glover, Craven
\cite[Theorem 2.1]{glo94} using the approximate subdifferential which
is constructed with the help of the lower Hadamard derivative.
Another Hadamard directional derivative is used in optimality conditions for
set-valued optimization, but it is constructed there with the help of the
limits of multifunctions in the sense of Painlev\'e-Kuratowski. This derivative
differs from the usual lower Hadamard derivative in the single-valued scalar case.

Let $\E$ be the $l$ dimensional Euclidean space and $S\subset\E$ a set in $\E$.
In the sequel we use the following notations:

\centerline{
$(-\infty,0)^k=\underbrace{(-\infty,0)\times (-\infty,0)\times\cdots\times
(-\infty,0)}_k$,}

\centerline{
$[0,\infty)^k=\underbrace{[0,\infty)\times [0,\infty)\times\cdots\times
[0,\infty)}_k$.
}

We consider the following cones.
The Bouligand tangent cone (or the contingent cone) \cite{bou32} of the set $S$
at the point $x\in\mathrm{cl}\, S$ is defined as follows:
\[
\begin{array}{c}
T(S,x):=\{u\in\E\mid\exists \{t_k\}\subset\R_>, t_k\to +0,
\exists \{u_k\}\subset\E, \\
u_k\to u\textrm{ such that }
x+t_ku_k\in S\textrm{ for all positive integers }k\}.
\end{array}
\]

The lower Hadamard conditional derivative \cite{dem98} of $f$ with respect to
the set $S$ at the point $x\in S$ in direction $u\in\E$ is defined as follows:
\[
\ld f (x;u;S)=\liminf_{(t,u^\prime)\to (+0,u),\, x+tu^\prime\in S}
t^{-1}(f(x+tu^\prime)-f(x)).
\]
By definition $\ld f (x;u;S)=+\infty$ if $u\notin T(S,x)$.
Somewhere it is called the contingent derivative.

Let $\bar f:\E\to\R\cup\{+\infty\}$ be the extension of $f$ such that
$\bar f(x)=+\infty$ for $x\in\E\setminus S$.
Then the lower conditional Hadamard derivative at the point $x\in S$ in direction $u\in\E$
could be defined as follows:
\[
\ld f (x;u;S)=\liminf_{(t,u^\prime)\to (+0,u)}
t^{-1}(\bar f(x+tu^\prime)-f(x)).
\]
If there exists the limit
\[
\d f (x;u;S)=\lim_{(t,u^\prime) \to (+0,u),\, x+tu^\prime\in S}
t^{-1}(f(x+tu^\prime)-f(x)),
\] then the function is called Hadamard differentiable
and $\d f (x;u)$ - its Hadamard derivative at $x\in S$ in direction $u\in\E$
(see, for instance, Demynov, Rubinov \cite{dem95}).

If the function $f$ attains its local minimum over the $S$ at the point
$\bar x$, then 
\[\ld f (x;u;S)\ge 0\quad\textrm{for all}\quad u\in\E\].

The paper is organized as follows: In section \ref{s2} we consider an exact penalty function and its conection with the optimality conditions.
In section \ref{s3} we obtain optimality conditions with non-strict inequalities. 
In section \ref{s4} we obtain optimality conditions with strict inequalities.

\section{Optimality and exact penalty function}
\label{s2}
Let $X\subset\E$ be a closed set and $f:\E \to\R$, $g_i:\E\to\R$, $i=1,2,...,m$,
$h_j:\E\to\R$, $j=1,2,...,q$ given functions. Consider the nonlinear
programming problem:


\vspace{0.5cm}
\noindent
$\textrm{Minimize }\;f(x)$ \\
$\textrm{subject to}\quad x\in X,\quad
g_i(x)\le0,\;i=1,2,...,m,\quad h_j(x)=0,\; j=1,2,...,q$. \hfill ${\rm{(P)}}$
\vspace{0.5cm}

Let $\bar x$ be a local minimizer of (P). Then there exists a number $\delta>0$
and a neighbourhood 
$$N_\delta(\bar x):=\{x\in\E\mid\norm{x-\bar x}\le\delta\}$$
of $\bar x$ such that $f(\bar x)\le f(x)$ for all feasible
$x\in N_\delta(\bar x)$.
Denote
\[
\begin{array}{rl}
h :=h_1^2+h_2^2+\cdots+h_q^2,\quad
G & :=\{x\in X \mid g_i(x)\le 0,\; i=1,2,...,m\},\\
G_\delta & :=G\cap N_\delta(\bar x).
\end{array}
\]
Then $\bar x$ is a global minimizer of the following problem:

\vspace{0.5cm}
\noindent
Minimize\quad $f(x)$\quad subject to \quad $x\in G_\delta$,\quad $h(x)=0$.\hfill (P$_1$)
\vspace{0.5cm}

For arbitrary constant $\gamma>0$ consider the penalty function defined on $G$
\[
F(x,\gamma):=f(x)+\gamma h(x)+0.5\norm{x-\bar x}^2.
\]

Denote
\[
S:=\{x\in G\mid h(x)=0\}.
\]

We prove under some hypotheses that there exists a positive integer $s$
such that each local minimizer of $f(x)$ over $S$ is a local minimizer of
$F(x,\gamma)$ over $G$ for all $\gamma>s$. 
Thus, we reduce the problem with inequality, equality and non-functional constraints to a problem with
inequality and non-functional ones. Then, the following question arises:
Is this reduction of the problem useful or useless? Our approach differs from
the customary practice when penalty functions are used. They usually
reduce the problem to unconstrained one. Taking into account the relation
between the original objective function and the new one, under our hypotheses,
the Lagrange multipliers in the Fritz John conditions in front of the
equality constraints are equal to zero. Thus, we obtain optimality conditions
which express in concrete form the standard Fritz John conditions.

In the next theorem we consider the function
\[
d(x):=\inf_{\norm{u}=1, u\in T(G,x)}\ld h(x;u;X)
\]
where $T(G,x)$ is the Bouligand tangent cone of the set $G$ at $x$.

\begin{theorem}\label{th-penalty}
Let $\bar x$ be a local minimizer  of the Problem {\rm (P)}, the set $X$ be
closed, the function $f$ be Lipschitz with a constant $L$ on
$N_\delta(\bar x)\cap X$ for some $\delta>0$, the functions $g_i$, $i=1,2,...,m$
be lower semicontinuous, and the functions $h_j$, $j=1,2,...,q$ be Hadamard
differentiable on $X$ and lower semicontinuous. Suppose that there exists $a>0$
with
\begin{equation}\label{6}
d(x)\le -a\quad\textrm{for all}\quad x\in N_\delta(\bar x)\setminus S.
\end{equation}
Then  there exists an integer $s$ such that $\bar x$ is a local minimizer
of $F(x,\gamma)$ on $G$ for all $\gamma>s$.
\end{theorem}
\begin{proof}
Without loss of generality we suppose that
the point $\bar x$ is a global minimizer of the problem (P$_1$).
We prove that there exists an integer $s$ such that $\bar x$ is a global
minimizer of $F(x,\gamma)$ on $G_\delta$ for all $\gamma>s$.
Assume the contrary that for all integers $k$ there exists $\gamma^k>k$
such that $\bar x$ is not a global minimizer of $F(x,\gamma^k)$ on $G_\delta$. 
Since $F$ is lower semicontinuous and $G_\delta$ is compact, we could suppose
that the minimum of $F(x,\gamma^k)$ on $G_\delta$ is attained at some point
$x^k$, $x^k\ne\bar x$. Therefore
\begin{equation}\label{p1}
f(x^k)+\gamma^k h(x^k)+0.5\norm{x^k-\bar x}^2=F(x^k,\gamma^k)\le F(\bar x,\gamma^k)=f(\bar x).
\end{equation}
Since the sequence $\{f(x^k)\}_{k=1}^\infty$ is bounded on $G_\delta$ and
$\gamma^k\to\infty$, then $\lim_{k\to\infty} h(x^k)=0$. By $x^k\in N_\delta(\bar x)$
there exists an accumulation point $x^*$ of the sequence $\{x^k\}$.
Without loss of generality $x^k\to x^*$. Since $G_\delta$ is closed, then $x^*\in G_\delta$.
By the lower semicontinuity of $h$ we have
\[
h(x^*)\le\liminf_{k\to\infty}h(x^k)=0.
\]
Using that $h$ is non-negative, we obtain $h(x^*)=0$.
Therefore $x^*$ is feasible for (P). By
\[
f(x^k)+0.5\norm{x^k-\bar x}^2\le f(x^k)+\gamma^k h(x^k)+0.5\norm{x^k-\bar x}^2\le f(\bar x)
\]
we obtain that
\[
f(x^*)+0.5\norm{x^*-\bar x}^2\le f(\bar x)\le f(x^*).
\]
Hence
$x^*\equiv\bar x$. Since every convergent subsequence of $\{x^k\}$ converges
to $\bar x$, then the whole sequence $\{x^k\}$ converges to $\bar x$.
It follows from the minimality of $x^k$ that
\begin{equation}\label{3}
\ld F(x^k,\gamma^k;u;G_\delta)\ge 0\quad\forall u\in \E
\end{equation}

Consider the following two cases:

1$^0$) 
Let us suppose that there exist an infinite number of points $x^k$ such that $x^k\notin S$. 
It follows from $d(x^k)\le -a$ that there exists
\[
u^k\in T(G,x^k) \quad\textrm{with}\quad
\norm{u^k}=1,\quad  h^\prime_H(x^k;u^k;X)\le -\frac{a}{2}.
\]
Let $k$ be fixed. Then
there exist $(t^k_n,u^k_n)\to(+0,u^k)$ when $n\to +\infty$ such that
$x^k+t^k_n u^k_n\in G$. By Hadamard differentiability of $h$ on $X$ we have
\begin{equation}\label{p2}
\lim_{n\to+\infty}\frac{h(x^k+t^k_n u^k_n)-h(x^k)}{t^k_n}=
h^\prime_H(x^k;u^k;X)\le -\frac{a}{2}.
\end{equation}
Hence $\norm{u^k_n}\le 2$ and
\begin{equation}\label{p3}
h(x^k+t^k_n u^k_n)-h(x^k)\le -\frac{a}{4}t^k_n
\end{equation}
for all sufficiently large $n$. Since $x^k+t^k_n u^k_n\in N_\delta(\bar x)$
for sufficiently large $n$ and $k$ because $x^k\to\bar x$, we have $x^k+t^k_n u^k_n\in G_\delta$
and $u^k\in T(G_\delta,x^k)$. Consider the difference
\[
\begin{array}{rrl}
D^k_n & := & F(x^k+t^k_n u^k_n,\gamma^k)-F(x^k,\gamma^k)=f(x^k+t^k_n u^k_n)-f(x^k)\\
&+ & \gamma^k(h(x^k+t^k_n u^k_n)-h(x^k))+\frac{1}{2}\norm{x^k+t^k_n u^k_n-\bar x}^2-
\frac{1}{2}\norm{x^k-\bar x}^2.
\end{array}
\]
By the Lipschitz property
\[
|f(x^k+t^k_n u^k_n)-f(x^k)|\le L t^k_n\norm{u^k_n}\le 2L t^k_n.
\]
Therefore
\[
D^k_n/t^k_n\le 2L-\frac{1}{4}a\gamma^k+\frac{1}{2}t^k_n\norm{u^k_n}^2+
\scalpr{x^k-\bar x}{u^k_n}\le 2L-\frac{1}{4} a\gamma^k+2t^k_n+2\delta
\]
Thus we have for all sufficiently large $k$ that
\[
\ld F(x^k,\gamma^k;u^k;G_\delta)\le\liminf_{n\to\infty}D^k_n/t^k_n<0.
\]
The last inequality  contradicts inequality (\ref{3}). Therefore, this case is impossible. 

2$^0$) All points $x^k$ except a finite number belong to $S$.
Without loss of generality $x^k\in S$ for all positive integers $k$. Hence
\[
f(x^k)+
\frac{1}{2}\norm{x^k-\bar x}^2\le F(x^k,\gamma^k)=\min_{x\in G_\delta}F(x,\gamma^k)
\le F(\bar x,\gamma^k)=f(\bar x)\le f(x^k).
\]
Therefore $x^k\equiv\bar x$, a contradiction with $x^k\ne\bar x$.
 
Thus,  $\bar x$ is a global minimizer of $F(x,\gamma)$ on
$G_\delta$ for all sufficiently large $\gamma>0$. Therefore
$\bar x$ is a local minimizer of $F(x,\gamma)$ on $G$ for these $\gamma$.
\end{proof}

\begin{remark}
Similar assumption to the constraint qualification (\ref{6}) appeared in 
Demyanov, Di Pillo, Facchinei \cite{dem98}. 
More conditions for local exact penalties in terms of the lower Hadamard
directional derivative could be found in Rosenberg \cite[Theorem 3]{ros84},
Ward \cite[Theorem 4.1]{war88}.
\end{remark}

\section{Optimality conditions with non-strict inequalities}
\label{s3}

For every feasible point $x$ of the problem (P) denote the set
of active constraints by
\[
I(x):=\{i\in\{1,2,...,m\}\mid g_i(x)=0\}.
\]
Suppose that $I(\bar x)=\{1,2,....,p\}$ where $p\le m$.

Besides the usual algebraic operations with infinities,
we accept that 
$$(\pm\infty)\cdot 0=0\cdot (\pm\infty)=0.$$

\begin{theorem}
{\bf (Necessary condition for a local minimum)}
Let $\bar x$ be a local minimizer of {\rm (P)} and the set $X$ be closed.
Suppose that the function $f$ is Lipschitz with a constant $L$
on $N_\delta(\bar x)\cap X$ for some $\delta>0$, the functions $h_j$,
$j=1,2,...,q$ are Hadamard differentiable and continuous,
the functions $g_i$, $i\in I(\bar x)$ are Hadamard differentiable and lower
semicontinuous, the functions $g_i$, $i\notin I(\bar x)$ are continuous.
Assume that there exists $a>0$ with $d(x)\le -a$ for all
$x\in N_\delta(\bar x)\setminus S$. Then, for every $u\in\E$
there exist
\[
\lambda=(\lambda_0,\lambda_1,...,\lambda_p)\in[0,\infty)^{p+1}, \;\;\lambda\ne 0
\]
such that
\begin{equation}\label{lagrange}
\lambda_0\ld F(\bar x;u;X)+\sum_{i=1}^p\lambda_i\d {(g_i)} (\bar x;u;X)\ge 0.
\end{equation}
If $\ld F (\bar x;u;X)=-\infty$, then $\lambda_0=0$.
If $\d {(g_i)} (\bar x;u;X)=-\infty$, then $\lambda_i=0$  $(i=1,2,...,p)$.
\end{theorem}

\begin{proof}
By Theorem \ref{th-penalty} there exists an integer $s$ such that $\bar x$
is a local minimizer of $F(x,\gamma)$ on $G$ for all $\gamma>s$.

1$^0$)  We prove that there is no $u\in\E$ such that
\begin{equation}\label{4}
\ld F (\bar x,\gamma;u;X)<0,\quad \d {(g_i)}(\bar x;u;X)<0,\; i\in I(\bar x)
\end{equation}
where $\gamma>s$ is fixed.
Assume that there exists $u\in\E$ which satisfies the system (\ref{4}) with
$\gamma>s$ fixed.
Therefore, there are sequences $\{t_k\}\subset (0,\infty)$, $t_k\to +0$,
$\{u_k\}\subset\E,$ $u_k\to u$ with $\bar x+t_ku_k\in X$ such that
\[
F(\bar x+t_ku_k,\gamma)<F(\bar x,\gamma), \quad
g_i(\bar x+t_ku_k)<g_i(\bar x)=0,\; i\in I(\bar x).
\]
It follows from the continuity of $g_i$ that
$g_i(\bar x+t_ku_k)<0$ for all sufficiently large $k$ because $g_i(\bar x)<0$
when $i\notin I(\bar x)$. Hence $\bar x$ is not a local minimizer of
$F(x,\gamma)$ on $G$, a contradiction.

2$^0$)  Let $u$ be arbitrary fixed. Denote
\[
\bar\alpha_0=\ld F (\bar x,\gamma;u;X),\;
\bar\alpha_i=\d {(g_i)}(\bar x;u;X), \; i=1,2,...,p,\;
\bar\alpha=(\bar\alpha_0,\bar\alpha_1,....,\bar\alpha_p).
\]
We have $\bar\alpha\notin(-\infty,0)^{p+1}$ and $(-\infty,0)^{p+1}$ is convex.
Applying the Separation Theorem we see that there exists
$\lambda=(\lambda_0,\lambda_1,...,\lambda_p)\ne 0$ such that
\[
\scalpr{\lambda}{\bar\alpha}\ge 0\quad\textrm{and}\quad\scalpr{\lambda}{\alpha}\le 0
\]
for all $\alpha\in(-\infty,0)^{p+1}$.  Let $\alpha=(\alpha_0,\alpha_1,\dots,\alpha_p)$ and    let $i\in\{0,1,...,p\}$ be arbitrary fixed.
Taking $\alpha_j\to 0$ when $j\ne i$ we get that $\lambda_i\alpha_i\le 0$
for all $\alpha_i\in(-\infty,0)$. Therefore $\lambda_i\ge 0$.
We obtain from $\scalpr{\lambda}{\bar\alpha}\ge 0$ that
\begin{equation}\label{5}
\lambda_0\ld F(\bar x,\gamma;u;X)+\sum_{i=1}^p\lambda_i\d {(g_i)} (\bar x;u;X)\ge 0.
\end{equation}
The inequality (\ref{5}) holds if among the components of $\bar\alpha$
there are infinities. Indeed, it follows from (\ref{4}) that at least one
$\bar\alpha_i>-\infty$. If there is  $i\in\{0,1,....,p\}$ such that
$\bar\alpha_i=+\infty$, then (\ref{5}) is trivially satisfied. We take
$\lambda_i=1$ for all such $i$.
For all $\bar\alpha_i\in[-\infty,+\infty)$ we choose $\lambda_i=0$.
If there is no $i\in\{0,1,....,p\}$ with $\bar\alpha_i=+\infty$,
then  we apply the
Separation Theorem again removing from $\bar\alpha$ all components
$\bar\alpha_i$ such that $\bar\alpha_i=-\infty$. Hence, we obtain the inequality
(\ref{5}) again taking $\lambda_i=0$ if $\bar\alpha_i=-\infty$.

\end{proof}

\section{Optimality conditions with strict inequalities}
\label{s4}
A feasible point $\bar x$ is called an isolated local minimizer
if there exist a positive real $A$ and a neighborhood $N\ni\bar x$ such that
\[
f(x)\ge f(\bar x)+ A\,\norm{x-\bar x}
\]
for all feasible points $x$ with $x\in N$.

\begin{theorem}
Let $\bar x$  be a feasible point for the problem (P). Suppose that there exist
\[
\lambda=(\lambda_0,\lambda_1,...,\lambda_p)\in\R^{1+p}
\]
with
\begin{equation}\label{strictlagrange}
\lambda_0\ld F(\bar x;u;X)+\sum_{i=1}^p\lambda_i\d {(g_i)} (\bar x;u;X)>0.
\end{equation}
for all directions $u\in\E$. Then $\bar x$ is an isolated local minimizer
of the problem {\rm (P)}.
\end{theorem}
\begin{proof}
1$^0$) First, we prove that $\bar x$ is an isolated local minimizer of the function $F(x,\gamma)$ on $G$
for all sufficiently large $\gamma>0$. Indeed, assume the contrary that there exist
$\gamma>0$ such that for every $\varepsilon>0$ there is
$x\in N_\varepsilon(\bar x)\cap G$ with
\[
F(x,\gamma)<F(\bar x,\gamma)+\varepsilon\norm{x-\bar x}.
\] 
Therefore,
for every sequence of positive numbers $\{\varepsilon_k\}$ converging to 0
there exists a sequence $\{x_k\}$ such that $x_k\to\bar x$ and
\[
F(x_k,\gamma)<F(\bar x,\gamma)+\varepsilon_k\norm{x_k-\bar x}.
\]
Denote $t_k=\norm{x_k-\bar x}$ and $u_k=(x_k-\bar x)/t_k$. We have that
\[
F(\bar x+t_k u_k,\gamma)-F(\bar x,\gamma)<\varepsilon_k t_k.
\]
Therefore $\ld F(\bar x,\gamma;u;X)\le 0$.
Because of $x_k\in G$ we have $g_i(x_k)\le g_i(\bar x)=0$ for all
$i\in I(\bar x)$. Therefore, $\d {(g_i)}(\bar x,\gamma;u;X)\le 0$.


According to $\lambda_i\ge 0$, $i=0,1,...,m$,  we obtain that
\[
\lambda_0\ld F(\bar x;u;X)+\sum_{i=1}^p\lambda_i\d {(g_i)} (\bar x;u;X)\le 0
\]
which contradicts the hypothesis of the theorem.

2$^0$) We prove $\bar x$ is an isolated local minimizer of (P).
Since $\bar x$ is an isolated local minimizer of   $F(x,\gamma)$ there exist
$A>0$ and $\delta>0$ such that
\[
f(x)+\gamma h(x)+0.5\norm{x-\bar x}^2\ge f(\bar x)+A\norm{x-\bar x}
\]
for all $x\in N_\delta(\bar x)\cap G$. The inequality is also satisfied for all  $\delta_1<\delta$.   Suppose that $x$ is feasible for (P).
We choose $\delta_1$ such that $\delta_1<2A$. Taking into account that
$\norm{x-\bar x}\le\delta_1$ we obtain
\[
f(x)\ge f(\bar x)+(A-0.5\delta_1)\norm{x-\bar x}\quad\forall x\in G_{\delta_1}
\textrm{ with } h(x)=0
\]
which implies that $\bar x$ is an isolated local minimizer of (P).
\end{proof}

\begin{theorem}
\label{th4.2}
Additionally to the hypothesis of Theorem \ref{th-penalty} we suppose that
$\bar x$ is an isolated local minimizer. Then there exists an integer $s$
such that such that $\bar x$ is an isolated local minimizer of $F(x,\gamma)$
on $G$ for all $\gamma>s$.
\end{theorem}
\begin{proof}
We use the arguments of Theorem \ref{th-penalty}. Suppose that there exist
$A>0$ and $\delta>0$ such that
\[
f(x)\ge f(\bar x)+A\norm{x-\bar x}
\]
for all feasible points $x$ of (P) such that $\norm{x-\bar x}\le\delta$.
We prove that there exists an integer $s$ with
\[
F(x,\gamma)\ge F(\bar x,\gamma)+A\norm{x-\bar x}
\]
for all feasible points $x$ of $F(x,\gamma)$ such that $\gamma>s$ and $\norm{x-\bar x}\le\delta$.
Assume the contrary that for every integer $k$ there exists $\gamma^k>k$
such that $\bar x$ is not a global minimizer of the function
$G(x,\gamma^k)=F(x,\gamma^k)-A\norm{x-\bar x}$ over $G_\delta$.
Since $G$ is lower semicontinuous and $G_\delta$ is compact the global
minimum of $G(x,\gamma^k)$ is attained at some point $x^k$. Therefore
\[
f(x^k)+\gamma^k h(x^k)+0.5\norm{x^k-\bar x}^2-A\norm{x^k-\bar x}
=F(x^k,\gamma^k)-A\norm{x^k-\bar x}\le F(\bar x,\gamma^k)=f(\bar x).
\]
Using the argument of Theorem \ref{th-penalty} we obtain that the sequence
$\{x^k\}$ converges to $\bar x$. It follows from the minimality of $x^k$
that
\[
\ld F(x^k,\gamma^k;u;G_\delta)\ge A\scalpr{x^k-\bar x}{u}\quad\forall u\in \E
\]
It follows from the arguments of Theorem \ref{th-penalty} that the case
there exists an infinite number of points $x^k$ such that $x^k\notin S$
is impossible. In the other case when $x^k\in S$ for all sufficiently large $k$
we have that $x^k\equiv\bar x$ for all sufficiently large $k$.
Therefore the claim of the theorem holds.
\end{proof}

The following lemma is a particular case of the Strict Separation Theorem. We prove it because the proof is very simple.

\begin{lemma}\label{lema1}
The following two conditions are equivalent:

\begin{equation}\label{14}
(a,b)\notin [-\infty,0]\times [-\infty,0]^p
\end{equation}
and
\begin{equation}\label{15}
\begin{array}{r}
\exists (\lambda,\mu)\in[0,+\infty)\times[0,\infty)^p: \lambda a+\scalpr{\mu}{b}>0,\\
\lambda=0\; \textrm{if}\; a=-\infty,
\; \mu_i=0\;\textrm{if}\; b_i=-\infty.
\end{array}\end{equation}
\end{lemma}

{\begin{proof}
Let (\ref{14}) hold. We prove that condition (\ref{15}) is satisfied. Assume the contrary
that
\[
\lambda a+\scalpr{\mu}{b}\le 0
\]
for all $(\lambda,\mu)\in[0,+\infty)\times[0,\infty)^p$. Take $\mu=0$.
Then we have $\lambda a\le 0$ for all $\lambda>0$. Therefore $a\le 0$.
Using similar arguments we obtain that $b_i\le 0$ for all
$i=1,2,...,p$. 
This result contradicts to the condition (\ref{14}).

The claim is trivially satisfied if among $a$, $b_i$ $(i=1,2,...,p)$,
 there are infinities. For all $a$, $b_i$ equal to $-\infty$
we put the respective multipliers equal to 0. Then we remove all these ones equal
to $-\infty$ and apply the reasonings given above.

Let (\ref{15}) be satisfied. We prove condition (\ref{14}). Assume the
contrary that
\[
(a,b)\in [-\infty,0]\times [-\infty,0]^p.
\]
Therefore
\[
\lambda a+\scalpr{\mu}{b}\le 0\quad\forall\;
(\lambda,\mu)\in[0,+\infty)\times[0,\infty)^p
\]
which contradicts (\ref{15}).
\end{proof}

Denote by $C(\bar x)$ the cone
\[
C(x)=\{u\in T(X,x)\mid \d {(g_i)} (x;u;X)\le 0, i\in I(x)\}.
\]

The following regulatity condition is usually called the Abadie constraint qualification when the functions are differentiable:
\[
 T(G,\bar x)=C(\bar x).
\]

\begin{theorem}
Let the functions $f$, $g$, $h$ be Hadamard differentiable and $\bar x$ be
an isolated local minimimizer of {\rm (P)}. Suppose that the Abadie constraint qualification holds.
 Then, for every $u\in\E$
there exist
\[
\lambda=(\lambda_0,\lambda_1,...,\lambda_p)\in[0,\infty)^{p+1}
\]
such that
\begin{equation}
\lambda_0\ld F(\bar x;u;X)+\sum_{i=1}^p\lambda_i\d {(g_i)} (\bar x;u;X)> 0.
\end{equation}
If $\ld F (\bar x;u;X)=-\infty$, then $\lambda_0=0$.
If $\d {(g_i)} (\bar x;u;X)=-\infty$, then $\lambda_i=0$  $(i=1,2,...,p)$.
\end{theorem}
\begin{proof}
We prove that there does not exist $u\in\E$ such that
\begin{equation}\label{system}
\ld F (\bar x,\gamma;u;X)\le 0,\quad \d {(g_i)}(\bar x;u;X)\le 0,\; i\in I(\bar x)
\end{equation}
for all sufficiently large $\gamma>0$. 

Assume the contrary. Let $u\in C(\bar x)$. It follows from $T(G,\bar x)=C(\bar x)$ that there exist
sequences $t_k\to+0$, $u_k\to u$ with $\bar x+t_k u_k\in G$. Therefore
$g_i(\bar x+t_k u_k)\le 0$, $i\in I(\bar x)$. It is clear that
$g_i(\bar x+t_k u_k)< 0$ for all sufficiently large $k$ when
$i\notin I(\bar x)$. Therefore $\bar x+t_k u_k$ satisfies the inequality constraints for all
sufficiently large integers $k$. By Theorem \ref{th4.2} we conclude that $\bar x$ is an isolated local minimizer of $F(x,\gamma)$ for all sufficiently large $\gamma$.  Using that
\[
F(\bar x+t_k u_k)\ge F(\bar x)+A t_k\norm{u_k}
\]
for all sufficiently large integers $k$, we obtain that
\[
\frac{F(\bar x+t_k u_k)-F(\bar x)}{t_k}\ge A \norm{u_k}.
\]
Taking the limits as $k$ approaches $+\infty$ we get
$\ld F(\bar x;u;X)\ge A\norm{u}>0$
which contradicts the inequalities (\ref{system}). 

Then the theorem follows by Lemma \ref{lema1} from
the incompatibility of the system (\ref{system}).
\end{proof}

\end{document}